\documentclass[12pt,reqno]{amsart}
\usepackage{amsmath}
\usepackage{amssymb}
\usepackage[left=3cm,top=3cm,right=3cm,bottom=3cm]{geometry}
\usepackage{epsfig}

\newcommand{\remove}[1]{}

\begin{document}
\newtheorem{theorem}{Theorem}[section]
\newtheorem{lemma}[theorem]{Lemma}
\newtheorem{definition}[theorem]{Definition}
\newtheorem{remark}[theorem]{Remark}
\newtheorem{conjecture}[theorem]{Conjecture}
\newtheorem{proposition}[theorem]{Proposition}
\newtheorem{algorithm}[theorem]{Algorithm}
\newtheorem{corollary}[theorem]{Corollary}
\newtheorem{observation}[theorem]{Observation}
\newtheorem{problem}[theorem]{Open Problem}
\newcommand{\noin}{\noindent}
\newcommand{\ellb}{\bar{\ell}}
\newcommand{\ind}{\indent}
\newcommand{\al}{\alpha}
\newcommand{\om}{\omega}
\newcommand{\pp}{\mathcal P}
\newcommand{\ppp}{\mathfrak P}
\newcommand{\R}{{\mathbb R}}
\newcommand{\N}{{\mathbb N}}
\newcommand{\Z}{{\mathbb Z}}
\newcommand\eps{\varepsilon}
\newcommand{\E}{\mathbb E}
\newcommand{\Var}{\mathbb{ V}\mathrm{ar}}
\newcommand{\Prob}{\mathbb{P}}
\newcommand{\pl}{\textrm{C}}
\newcommand{\dang}{\textrm{dang}}
\renewcommand{\labelenumi}{(\roman{enumi})}
\newcommand{\bc}{\bar c}
\newcommand{\cal}[1]{\mathcal{#1}}
\newcommand{\G}{{\cal G}}
\renewcommand{\P}{{\cal P}}

\newcommand{\bel}[1]{\be\lab{#1}}
\newcommand{\ee}{\end{equation}}
\newcommand{\be}{\begin{equation}}
 \newcommand\eqn[1]{(\ref{#1})}

\title{Vertex-pursuit in random directed acyclic graphs}

\author{Anthony Bonato}
\address{Department of Mathematics, Ryerson University, Toronto, ON, Canada, M5B 2K3}
\email{\tt abonato@ryerson.ca}

\author{Dieter Mitsche}
\address{Department of Mathematics, Ryerson University, Toronto, ON, Canada, M5B 2K3}
\email{\texttt{dmitsche@ryerson.ca}}

\author{Pawe\l{} Pra\l{}at}
\address{Department of Mathematics, Ryerson University, Toronto, ON, Canada, M5B 2K3}
\email{\texttt{pralat@ryerson.ca}}

\thanks{The authors gratefully acknowledge support from NSERC, Mprime, and Ryerson University.}


\maketitle

\begin{abstract}
We examine a dynamic model for the disruption of information flow in hierarchical social networks by considering the vertex-pursuit game Seepage played in directed acyclic graphs (DAGs). In Seepage, agents attempt to block the movement of an intruder who moves downward from the source node to a sink. The minimum number of such agents required to block the intruder is called the green number. We propose a generalized stochastic model for DAGs with given expected total degree sequence. Seepage and the green number is analyzed in stochastic DAGs in both the cases of a regular and power law degree sequence. For each such sequence, we give asymptotic bounds (and in certain instances, precise values) for the green number.
\end{abstract}

\section{Introduction}\label{sec:intro}

The on-line social network Twitter is a well known example of a complex real-world network with over $300$ million users. The topology of Twitter network is highly directed, with each user following another (with no requirement of reciprocity). By focusing on a popular user as a source (such as Lady Gaga or Justin Bieber, each of whom have over 11 million followers~\cite{twitter}), we may view the followers of the user as a certain large-scale \emph{hierarchical social network}. In such networks, users are organized on ranked levels below the source, with links (and as such, information) flowing from the source downwards to sinks. We may view hierarchical social networks as \emph{directed acyclic graphs}, or \emph{DAGs} for short. Hierarchical social networks appear in a wide range of contexts in real-world networks, ranging from terrorist cells to the social organization in companies; see, for example~\cite{almen,farley2,gupte,ikeda,lopez}.

In hierarchical social networks, information flows downwards from the source to sinks. Disrupting the flow of information may correspond to halting the spread of news or gossip in on-line social network, or intercepting a message sent in a terrorist network. How do we disrupt this flow of information while minimizing the resources used? We consider a simple model in the form of a vertex-pursuit game called Seepage introduced in \cite{CFFMN}. Seepage is motivated by the 1973 eruption of the Eldfell volcano in Iceland. In order to protect the harbour, the inhabitants poured water on the lava in order to solidify it and thus, halt its progress. The game has two players, the \emph{sludge} and a set of \emph{greens} (note that one player controls all the greens), a DAG with one source (corresponding to the top of the volcano) and many sinks (representing the lake). The players take turns, with the sludge going first by contaminating the top node (source). Then it is the greens' turn, and they choose some non-protected, non-contaminated nodes to protect. On subsequent rounds the sludge moves a non-protected node that is adjacent (that is, downhill) to the node the sludge is currently occupying and contaminates it; note that the sludge is located at a single node in each turn. The greens, on their turn, proceed as before; that is, choose some non-protected, non-contaminated nodes to protect. Once protected or contaminated, a node stays in that state to the end of the game. The sludge wins if some sink is contaminated; otherwise the greens win, that is, if they erect a cutset of nodes which separates the contaminated nodes from the sinks. The name ``Seepage'' is used because the rate of contamination is slow. The game is related to vertex-pursuit games such as Cops and Robbers (for an introduction and further reading on such games, see~\cite{AB1}), although the greens in our case need not move to neighbouring nodes. For an example, see the DAG in Figure~\ref{seepageex1}. (We omit orientations of directed edges in the figure, and assume all edges point from higher nodes to lower ones.)
\begin{figure} [h]
\begin{center}
\epsfig{figure=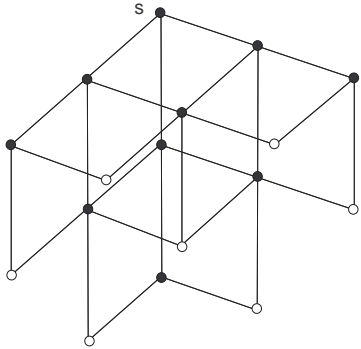,width=2in,height=2in}
\caption{A DAG where $2$ greens are needed to win. The white nodes are the sinks.}\label{seepageex1}
\end{center}
\end{figure}

To obtain the results in this paper, a number of different winning strategies are employed by the two players. In some cases one of the two players can play arbitrarily (at least up to some point), whereas in other cases the optimal strategy is simply a ``greedy'' one (for example, when the greens protect neighbours as close as possible to the current position of the sludge). In some other cases, much more sophisticated strategies have to be applied.

To date the only analysis of Seepage was in \cite{CFFMN}, which presented results for DAGs. Seepage may be extended to certain directed graphs with cycles, although we do not consider this variation here (see also Section~\ref{directions}). In \cite{CFFMN}, a characterization was given of directed trees where one green has a winning strategy, and bounds were given on the number of greens needed to win in truncated products of paths. See also Chapter~9 of \cite{AB1}.

Seepage displays some interesting similarities to an approach used in mathematical counterterrorism, where cut sets in partially ordered sets (which are just a special kind of DAG) are used to model the disruption of terrorist cells. As described in Farley~\cite{farley2,farley1}, the maximal elements of the poset are viewed as the leaders of the terrorist organization, who submit plans down via the edges to the nodes at the bottom (the foot soldiers or minimal nodes). Only one messenger needs to receive the message for the plan to be executed. Farley considered finding minimum-order sets of elements in the poset, which when deleted, disconnect the minimal elements from the maximal one (that is, find a \emph{minimum cut}). We were struck by the similarities in the underlying approaches in~\cite{CFFMN} and~\cite{farley2,farley1}; for example, in Seepage the greens are trying to prevent the sludge from moving to the sinks by blocking nodes. The main difference is that Seepage is ``dynamic'' (that is, the greens can move, or choose new sets of nodes each time-step), while the min-cut-set approach is ``static'' (that is, find a cutset in one time-step). Seepage is perhaps a more realistic model of counterterrorism, as the agents do not necessarily act all at once but over time. However, in both approaches deterministic graphs are used.

We note that a stochastic model was presented for so-called \emph{network interdiction} in~\cite{interdiction}, where the task of the interdictor is to find a set of edges in a weighted network such that the removal of those edges would maximally increase the cost to an evader of traveling on a path through the network. A stochastic model for complex DAGs was given in~\cite{chayes}. For more on models of on-line social networks and other complex networks, see~\cite{bonato}.

Our goal in the present article is to analyze Seepage and the green number when played on a random DAG as a model of disrupting a given hierarchical social network. We focus on mathematical results, and give a precise formulation of our random DAG model in Section~\ref{sec:definitions}. Our model includes as a parameter the total degree distribution of nodes in the DAG. This has some similarities to the $G(\mathbf{w})$ model of random graphs with expected degree sequences (see~\cite{clbook}) or the pairing model (see~\cite{wormald}). We study two cases: regular DAGs (where we would expect each level of the DAG to have nodes with about the same out-degree), and power law DAGs (where the degree distribution is heavy tailed, with many more low degree nodes but a few which have a high degree). Rigorous results are presented for regular DAGs in Theorem~\ref{thm:main_d-reg}, and for power law DAGs in Theorem~\ref{thm:main_power-law}. An overview of the main results is given in Section~\ref{mainr}.

Throughout, $G$ will represent a finite DAG. For background on graph theory, the reader is directed to~\cite{diestel,west}. The total degree of a vertex is the sum of its in- and out-degrees. Additional background on Seepage and other vertex-pursuit games may be found in~\cite{AB1}.

\section{Definitions}\label{sec:definitions}

We denote the natural numbers (including $0$) by $\mathbb{N}$, and the positive integers and real numbers by $\mathbb{N}^+$ and $\mathbb{R}^+$, respectively. For an event $A$ on a probability space, we let $\mathbb{P}(A)$ denote the probability of $A$. Given a random variable $X$, we let $\mathbb{E}(X)$ and $\mathbb{V}\mathrm{ar}(X)$ be the expectation and the variance of $X$, respectively.

We now give a formal definition of our vertex-pursuit game.  Fix $v \in V(G)$ a node of $G$. We will call $v$ the \emph{source}. For $i \in \N $ let
$$
L_i = L_i(G,v) = \{ u \in V(G) : \mathrm{dist}(u,v) = i \},
$$
where $\mathrm{dist}(u,v)$ is the distance between $u$ and $v$ in $G.$
In particular, $L_0 = \{v \}$. For a given $j \in \N^+$ and $c \in \R^+$, let $\G(G,v,j,c)$ be the game played on graph $G$ with the source $v$ and the \emph{sinks} $L_j$. The game proceeds over a sequence of discrete time-steps. Exactly
$$
c_t=\lfloor ct \rfloor- \lfloor c(t-1) \rfloor
$$
new nodes are protected at time-step $t$. (In particular, at most $ct$ nodes are protected by the time $t$.) Note that if $c$ is an integer, then exactly $c$ nodes are protected at each time-step, so this is a natural generalization of Seepage. To avoid trivialities, we assume that $L_j \neq \emptyset$.

The \emph{sludge} starts the game on the node $v_1=v$. The second player, the \emph{greens}, can protect $c_1 = \lfloor c \rfloor$ nodes of $G \setminus \{v\}$. Once nodes are protected they will stay protected to the end of the game. At time $t \ge 2$, the sludge makes the first move by sliding along a directed edge from $v_{t-1}$ to $v_t$, which is an out-neighbour of $v_{t-1}$. After that the greens have a chance to protect another $c_t$ nodes. Since the graph is finite and acyclic, the sludge will be forced to stop moving, and so the game will eventually terminate. If he reaches any node of $L_j$, then the sludge wins; otherwise, the greens win.

If $c = \Delta(G)$ (the maximum out-degree of $G$), then the game $\G(G,v,j,c)$ can be easily won by the greens by protecting of all neighbours of the source. Therefore, the following graph parameter, \emph{the green number}, is well defined:
$$
g_j(G,v) = \inf \{ c \in \R^+ : \G(G,v,j,c) \text{ is won by the greens} \}.
$$
It is clear that for any $j \in \N_+$ we have $g_{j+1}(G,v) \le g_j(G,v)$.

\subsection{Random DAG model}

There are two parameters of the model: $n \in \N^+$ and an infinite sequence $$\mathbf{w}=(w_1, w_2, \ldots)$$ of non-negative integers. Note that the $w_i$'s may be functions of $n$. The first layer (that is, the source) consists of one node: $L_0 = \{v\}$. The next layers are recursively defined. For the inductive hypothesis, suppose that all layers up to and including the layer $k$ are created, and let us label all nodes of those layers. In particular,
$$
L_k = \{ v_{d_{k-1}+1}, v_{d_{k-1}+2}, \ldots, v_{d_k} \},
$$
where $d_k = \sum_{i=0}^k |L_i|.$ We would like the nodes of $L_k$ to have a total degree with the following distribution $(w_{d_{k-1}+1}, w_{d_{k-1}+2}, \ldots, w_{d_k})$. However, it can happen that some node $v_i \in L_k$ has an in-degree $\deg^-(v_i)$ already larger than $w_i$, and so there is no hope for the total degree of $w_i$. If this is not the case, then the requirement can be easily fulfilled. As a result, $\mathbf{w}$, the desired degree distribution, will serve as a (deterministic) lower bound for the actual degree distribution we obtain during the (random) process.

Let $S$ be a new set of nodes of cardinality $n$. All directed edges that are created at this time-step will be from the layer $L_k$ to a random subset of $S$ that will form a new layer $L_{k+1}$. Each node $v_i \in L_k$ generates $\max \{ w_i - \deg^-(v_i), 0\}$ random directed edges from $v_i$ to $S$. Therefore, we generate $$e_k = \sum_{v_i \in L_k} \max \{ w_i - \deg^-(v_i), 0\}$$ random edges at this time-step. The destination of each edge is chosen uniformly at random from $S$. All edges are generated independently, and so we perform $e_k$ independent experiments. The set of nodes of $S$ that were chosen at least once forms a new layer $L_{k+1}$. Note that it can happen that two parallel edges are created during this process. However, this is a rare situation for sparse random graphs we are going to investigate in this paper. Hence, our results on the green number will also hold for a slightly modified process which excludes parallel edges.

\section{Main results}\label{mainr}

In this paper, we focus on two specific sequences: regular and power law. We will describe them both and state the main results in the next two subsections. We consider asymptotic properties of the model as $n\rightarrow \infty$. We say that an event in a probability space holds \emph{asymptotically almost surely} (\emph{a.a.s.}) if its probability tends to one as $n$ goes to infinity.

\subsection{Random regular DAGs}

We consider a constant sequence; that is, for $i \in \N^+$ we set $w_i = d$, where $d \ge 3$ is a constant. In this case, we refer to the stochastic model as \emph{random $d$-regular DAGs}. Since $w_i=d$, observe that $|L_j| \leq d(d-1)^{j-1}$ (deterministically) for any $j$, since at most $d(d-1)^{j-1}$ random edges are generated when $L_j$ is created. We will write $g_j$ for $g_j(G,v)$ since the graph $G$ is understood to be a $d$-regular random graph, and $L_0=\{v\}=\{v_1 \}$.

\begin{theorem}\label{thm:main_d-reg}
Let $\omega=\omega(n)$ be any function that grows (arbitrarily slowly) as $n$ tends to infinity. For the random $d$-regular DAGs, we have the following.
\begin{itemize}
\item [(i)] A.a.s.\ $g_1 = d$.
\item [(ii)] If $2 \le j = O(1)$, then a.a.s.\
$$
g_j = d-2+\frac {1}{j} .
$$
\item [(iii)] If $\omega \le j \le \log_{d-1} n - \omega \log \log n$, then a.a.s.\
$$
g_j = d-2.
$$
\item [(iv)] If $\log_{d-1} n - \omega \log \log n \le j \le \log_{d-1}n- \frac{5}{2} s \log_2 \log n + \log_{d-1} \log n - O(1)$ for some $s \in \N^+$, then a.a.s.\
$$
d-2-\frac {1}{s} \le g_j \le d-2.
$$
\item [(v)] Let $s \in \N^+$, $s \ge 4$. There exists a constant $C_s >0 $ such that if  $j \ge \log_{d-1}n + C_s$, then a.a.s.\
$$
g_j \le d-2-\frac {1}{s}.
$$
\end{itemize}
\end{theorem}

The whole Section~\ref{sec:d-reg} is devoted to prove this theorem. Theorem~\ref{thm:main_d-reg} tells us that the green number is slightly bigger than $d-2$ if the sinks are located near the source, and then it is $d-2$ for a large interval of $j$. Later, it might decrease slightly since an increasing number of vertices have already in-degree $2$ or more, but only for large $j$ (part (v)) we can prove  better upper bounds than $d-2.$ One interpretation of this fact is that the resources needed to disrupt the flow of information is in a typical regular DAG is (almost) independent of $j$, and relatively low (as a function of $j$).

\subsection{Random power law DAGs}

We have three parameters in this model: $\beta>2$, $d>0$, and $0 < \alpha < 1$. For a given set of parameters, let
$$
M=M(n)=n^{\alpha}, \mbox{\ \ \ }
i_0=i_0(n)=n \left(\frac{d}{M}  \frac{\beta-2}{\beta-1}\right)^{\beta-1},$$
and
$$
c=\left(\frac{\beta-2}{\beta-1}\right) dn^{\frac{1}{\beta-1}}.
$$
Finally, for $i \ge 1$ let
$$
w_i = c (i_0+i-1)^{-\frac{1}{\beta-1}}.
$$
In this case, we refer to the model as \emph{random power law DAGs}.

We note that the sequence $\mathbf{w}$ is decreasing (in particular, the source has the largest expected degree). Moreover, the number of coordinates that are at least $k$ is equal to
$$
n \left( \frac {\beta-2}{\beta-1}  \frac {d}{k} \right)^{\beta-1} - i_0 = (1+o(1)) n \left( \frac {\beta-2}{\beta-1}  \frac {d}{k} \right)^{\beta-1} = \Theta(n k^{-\beta+1}),
$$
and hence the sequence follows a power-law with exponent $\beta$. From the same observation it follows that the maximum value is
$$
w_1 = c i_0^{-\frac{1}{\beta-1}} = M.
$$
Finally, the average of the first $n$ values is
$$
\frac {c}{n} \sum_{i=i_0}^{i_0+n-1} i^{- \frac{1}{\beta-1}} = (1+o(1)) \frac {c}{n} \left(\frac {\beta-1}{\beta-2} \right) n^{1-\frac{1}{\beta-1}} = (1+o(1)) d,
$$
since $M=o(n)$.

Our main result on the green number $g_j=g_j(G,v)$ in the case of power law sequences is the following.
\begin{theorem}\label{thm:main_power-law}
Let
$$
\gamma = d^{\beta-1} \left( \frac {\beta-2}{\beta-1} \right)^{\beta-2} \left( \left(1+ \left( d \frac{\beta-2}{\beta-1} \right)^{1-\beta} \right)^{\frac{\beta-2}{\beta-1}}-1\right)
$$
if $\frac {1}{\alpha} - \beta + 3 \in \N^+ \setminus \{1,2\}$, and $\gamma = 1$ otherwise.
Let $j_1$ be the largest integer satisfying $j_1 \le \max\{ \frac {1}{\alpha} - \beta + 3, 2\}$. Let $j_2 = O(\log \log n)$ be the largest integer such that
$$
d^{\beta-1} \left(\frac{\gamma}{d^{\beta-1}} n^{\alpha(j_1-1)-1} \right)^{\left(\frac{\beta-2}{\beta-1}\right)^{j_2-j_1}} \le (\omega \log \log n)^{- \max\{ 2, (\beta-1)^2\}}.
$$
Finally, let
$$
\xi = \left(\frac {\beta-2}{\beta-1}\right) d \left( \left( \frac {d(\beta-2)}{\beta-1} \right)^{\beta-1} + 1 \right)^{-\frac {1}{\beta-1}}.
$$

Then, for $1 \leq j \le j_2-1$ we have that a.a.s.
\begin{equation}\label{moo}
(1+o(1)) \bar{w}_{j} \le g_j \le (1+o(1)) \bar{w}_{j-1},
\end{equation}
where $\bar{w}_{0}=\bar{w}_{1}=M$, for $2 \leq j < \frac{1}{\alpha} - \beta +3$,
$$
\bar{w}_{j} =
\begin{cases}
n^{\alpha} & \text{ if } 2 \le j < \frac {1}{\alpha} - \beta + 2 \\
\xi n^{\alpha} & \text{ if } 2 \le j = \frac {1}{\alpha} - \beta + 2 \\
\left( \frac{\beta-2}{\beta-1} \right) d n^{\frac {1-\alpha(j-1)}{\beta -1}} & \text{ if } \frac {1}{\alpha} - \beta + 2 < j < \frac {1}{\alpha} - \beta + 3 \text{ and } j \ge 2,
\end{cases}
$$
and for $j_1 \leq  j \leq j_2-1$,
$$
\bar{w}_{j} = \left(\frac {\beta-2}{\beta-1} \right) \left(\frac{\gamma}{d^{\beta-1}} n^{\alpha(j_1-1)-1} \right)^{-\left(\frac{\beta-2}{\beta-1}\right)^{j-j_1}/(\beta-1)}.
$$
\end{theorem}
In the power law case, Theorem~\ref{thm:main_power-law} tells us that the green number is smaller for large $j$. This reinforces the view that intercepting a message in a hierarchical social network following a power law is more difficult close to levels near the source.
\section{Proofs for random $d$-regular DAGs}\label{sec:d-reg}

Before analyzing the game on random $d$-regular DAGs, we need a few lemmas. We will be using the following version of a well-known Chernoff bound.
\begin{lemma}[\cite{JLR}]\label{lem:Chernoff}
Let $X$ be a random variable that can be expressed as a sum $X=\sum_{i=1}^{n} X_{i}$ of independent random indicator variables where
$X_i$ is a Bernoulli random variable with success probability $p_i$ with (possibly) different $p_{i}=\mathbb{P} (X_{i} = 1)= \mathbb{E}X_{i} > 0$. Then the following holds for $t \ge0$:
\begin{align}
\mathbb{P} (X \ge\mathbb{E }X + t)  &  \le\exp\left(  - \frac{t^{2}}{2(\mathbb{E }X+t/3)} \right)  ,\label{eq:Ch1}\\
\mathbb{P} (X \le\mathbb{E }X - t)  &  \le\exp\left(  - \frac{t^{2}}{2\mathbb{E }X} \right)  \label{eq:Ch2}.
\end{align}
In particular, if $\varepsilon\le3/2$, then
\begin{align}\label{eq:Ch3}
\mathbb{P} (|X - \mathbb{E }X| \ge\varepsilon\mathbb{E }X)  &  \le2 \exp\left(  - \frac{\varepsilon^{2} \mathbb{E }X}{3} \right)  .
\end{align}
\end{lemma}
We will start by proving the threshold for appearance of vertices of in-degree $k$.

\begin{lemma}\label{lem:layersize}
Let $\omega=\omega(n)$ be any function that grows (arbitrarily slowly) as $n$ tends to infinity. Then a.a.s.\ the following properties hold.
\begin{itemize}
\item [(i)] $|L_j| = (1-o(1))d(d-1)^{j-1}$ for any $1 \leq j \leq \log_{d-1}n-\omega$.
\item [(ii)]  For all $k \geq 2$, let $j_k=\frac{k-1}{k}\log_{d-1}n.$ For every $v \in L_j$, we have that $\deg^-(v) < k$ if $j< j_k -\omega$, and $\deg^-(v) = k$ for some $v \in L_{j_k+\omega}$. In particular, the threshold for the appearance of vertices of in-degree $k$ is $j_k$.
\end{itemize}
\begin{itemize}
\item [(iii)] $|L_j|=d(d-1)^{j-1}$ for $1 \leq j \leq \frac12 \log_{d-1}n-\omega$.
\end{itemize}
\end{lemma}

\begin{proof}
For (i) note that the probability that a given vertex $v \in S$ has in-degree $k \ge 2$ at level $j \ge 1$ is at most
$$
{ d(d-1)^{j-1} \choose k} \left( \frac {1}{n} \right)^k = O \left( \frac{(d-1)^{jk}}{n^k} \right).
$$
Thus, the expected number of vertices of in-degree $k$ at level $j$ is $O(\frac{(d-1)^{jk}}{n^{k-1}})$ and, in particular, the expected number of vertices of in-degree $2$ or more at $L_j$ is $O(\frac{(d-1)^{2j}}{n})$. Set $\alpha_j = (d-1)^{\frac12 \log_{d-1}n-\frac{j}{2}}$. By Markov's inequality, with probability at least $1-\frac{1}{\alpha_j}$ we derive that
\begin{equation}\label{eq:bound_deg2}
\Big| \{v \in L_j : \deg^-(v) \ge 2 \} \Big| = O \left(\frac{(d-1)^{2j}\alpha_j}{n} \right)=O \left(\frac{(d-1)^{3j/2}}{\sqrt{n}} \right).
\end{equation}
Since
$$
\sum_{j=1}^{\log_{d-1}n-\omega} \frac{1}{\alpha_j} = (d-1)^{-\frac12 \log_{d-1}n} \sum_{j=1}
^{\log_{d-1}n-\omega} (d-1)^{j/2}= O \Big((d-1)^{-\frac{\omega}{2}} \Big)=o(1),
$$
we obtain that a.a.s.\ \eqref{eq:bound_deg2} holds for all values of $j \leq \log_{d-1}n-\omega$. Since we aim for a statement that holds a.a.s., we can assume for $j \leq \log_{d-1}n-\omega$ that
\begin{equation}\label{eq:bound_deg}
|L_{j+1}| = (d-1) |L_j| - O \left(\frac{(d-1)^{3j/2}}{\sqrt{n}} \right).
\end{equation}

 We prove (i) by strong induction. It follows from~\eqref{eq:bound_deg2} that $|L_1|=d$, so (i) holds for $j=1$. Suppose that (i) holds for all $i<j$; that is, $|L_i| = (1-o(1))d(d-1)^{i-1}$. By~\eqref{eq:bound_deg} and the inductive hypothesis (used recursively), we obtain that
\begin{eqnarray*}
|L_j| & = & (d-1)|L_{j-1}|-O \left(\frac{(d-1)^{3j/2}}{\sqrt{n}} \right) \\
& = & (d-1)|L_{j-1}| \left( 1- O \left(\frac{(d-1)^{j/2}}{\sqrt{n}} \right)\right) \\
&=& d(d-1)^{j-1} E,
\end{eqnarray*}
where
\begin{eqnarray*}
E&=&\prod_{j \leq \log_{d-1}n-\omega} \left(1-O\left(\frac{(d-1)^{j/2}}{\sqrt{n}}\right)\right).
\end{eqnarray*}
Note that
\begin{eqnarray*}
E&=&\exp\left(-\sum_{j \leq \log_{d-1}n -\omega}O\left(\frac{(d-1)^{j/2}}{\sqrt{n}}\right)\right) \\
&=&\exp\left(-O \Big((d-1)^{-\omega/2}\Big)\right)=(1-o(1)).
\end{eqnarray*}

\bigskip

We now prove (ii) and (iii). By part (i), the number of random edges $e_{j-1}$ emanating from $L_{j-1}$ is $(1-o(1))d(d-1)^{j-1}.$ When layer $j$ is created, these edges are joined to random vertices in the set $S=\{ s_1, s_2, \ldots, s_n\}$ of cardinality $n$. For any fixed $k \geq 2$ and any fixed layer $1 \leq j \leq \log_{d-1}n-\omega$, we define the indicator variable $I_i$ to be $1$ if $s_i$ has in-degree $k$ and $0$ otherwise, for $i=1, 2, \ldots,n$. Let $X=\sum_{i=1}^n I_i$.

As observed in part (i) of this proof,
\begin{eqnarray*}
\Prob(I_i=1) &=& \binom{e_{j-1}}{k} \left(\frac{1}{n}\right)^k \left(1-\frac{1}{n}\right)^{e_{j-1}-k}\\
&=& \binom{e_{j-1}}{k} \left(\frac{1}{n}\right)^k (1+o(1)) = \Theta \left(\frac{(d-1)^{jk}}{n^k}\right),
\end{eqnarray*}
and thus, $\E(X)=\Theta(\frac{(d-1)^{jk}}{n^{k-1}})$. By Markov's inequality, a.a.s.\ for $j \le \frac{k-1}{k}\log_{d-1}n-\omega$, no vertices of in-degree $k$ are present. By considering the case $k=2$, this shows that for $j \leq \frac12 \log_{d-1}n-\omega$, all vertices have in-degree $1$ a.a.s., and thus, for such $j$ we have $|L_j| = d(d-1)^{j-1}$ a.a.s. Hence, (iii) holds.

We find that
\begin{eqnarray*}
\Prob(I_i=1, I_{i'}=1) &=& \binom{e_{j-1}}{k}\binom{e_{j-1}-k}{k} \left(\frac{1}{n} \right)^{2k} \left(1-\frac{2}{n}\right)^{e_{j-1}-2k} \\
&=& \binom{e_{j-1}}{k}^2 \left(\frac{1}{n} \right)^{2k} (1+o(1)).
\end{eqnarray*}
Thus,
\begin{eqnarray*}
\E(X^2)&=&\sum_{i \neq i'} (\Prob(I_i=1,I_{i'}=1)+\sum_{i} \Prob(I_i=1) \\
&\leq& \sum_{i \neq i'} \left(\binom{e_{j-1}}{k}^2 \left(\frac{1}{n}\right)^{2k} (1+o(1))\right)+ \E(X) \\
&=& \left[ n \binom{e_{j-1}}{k} \left(\frac{1}{n}\right)^{k} \right]^2 (1+o(1)) + \E(X) \\
& = & (\E(X))^2(1+o(1))+\E(X).
\end{eqnarray*}
Hence, for $j = \frac{k-1}{k}\log_{d-1}n+\omega$, it follows from Chebyshev's inequality that
$$
\Prob(X=0) \leq \frac{\Var(X)}{(\E(X))^2} = \frac {\E(X^2)}{(\E(X))^2} - 1 = o(1),
$$
and hence, a.a.s.\ there are vertices of in-degree $k$, proving part (ii) of the lemma.
\end{proof}

The lemma is enough to prove the first two parts of the main theorem.

\begin{proof}[Proof of Theorem~\ref{thm:main_d-reg} (i), (ii), and the upper bound of (iii)]
By Lemma~\ref{lem:layersize}~(iii) for $j \le \frac 12 \log_{d-1} n - \omega$ the game is played on a tree. Part (i) is trivial, since the greens have to protect all vertices in $L_1$, or they lose.

To derive the upper bound of (ii), note that for $c=d-2+\frac {1}{j}$ we have that $c_j=d-1$ ($c_i = d-2$ for $1 \le i \le j-1$). The greens can play arbitrarily during the first $j-1$ steps, and then block the sludge on level $j$. If $j \ge \omega$, then we have that $d-2$ is an upper bound of $g_j$, and the upper bound of (iii) holds.

To derive the lower bound of (ii), note that if $d-2 \le c < d-2-\frac{1}{j}$, then exactly $d-2$ new vertices are protected at each time-step.  Without loss of generality, we may assume that the greens always protect vertices adjacent to the sludge (since the game is played on the tree, there is no advantage to play differently). No matter how the greens play, there is always at least one vertex not protected and the sludge can reach $L_j$.
\end{proof}

\bigskip

For a given vertex $v \in L_t$ and integer $j$, let us denote by $S(v,j)$ the subset of $L_{t+j}$ consisting of vertices at distance $j$ from $v$ (that is, those that are in the $j$-th level of the subgraph whose root is $v$). Let $N(v,j)=\sum_{i=1}^j S(v,i)$ be the subgraph of all vertices of depth $j$ pending at $v$. Call a vertex $u \in S$ \emph{bad} if $u \in N(v,j)$ and $u$ has in-degree at least $2$ (recall, that $S$ is a set of $n$ vertices used in the process of generating a random graph). Let $X(v,j)$ be the total number of bad vertices in $N(v,j)$. In the next lemma, we estimate $X(v,j)$.

\begin{lemma}\label{lem:badvertices}
A.a.s.\ the following holds for some large enough constant $C' > 0.$ For any $v \in L_t$, where $t \leq \log_{d-1}n-\omega$, and any $j$ such that $\frac{(d-1)^{t+2j}}{n} \leq \log n$,
$$
X(v,j) \leq C' \log n.
$$
\end{lemma}
\begin{proof}
Fix $v \in L_t$ and let $j$ be the maximum integer satisfying $\frac{(d-1)^{t+2j}}{n} \leq \log n$. Since there are $O(n)$ possible vertices to consider, it is enough to show that the bound holds with probability $1-o(n^{-1})$.

For $u \in S$, let $I_u(v,i)$ ($1 \le i \le j$) be the event that $u \in S(v,i)$ and $u$ is bad. In order for $u$ to be in $S(v,i)$, $u$ must receive at least one edge from a vertex in $S(v,i-1)$, and in order to be bad it must have at least one more edge from either $S(v,i-1)$ or from another vertex at layer $L_{t+i-1}$. Thus,
$$
\Prob(I_u(v,i)) = \frac{O((d-1)^i)}{n}\frac{O((d-1)^{t+i})}{n}= O\left(\frac{(d-1)^{t+2i}}{n^2}\right),
$$
since there are $O((d-1)^i)$ edges emanating from $S(v,i-1)$, and there are $O((d-1)^{t+i})$ edges emanating from $L_{t+i-1}$. Letting $I_u=I_u(v,i)$ being the corresponding indicator variable, we have that
$$
\E \left(\sum_{u \in S} I_u \right) = O \left(\frac{(d-1)^{t+2i}}{n} \right).
$$
Note that $\Prob(I_u = 1 ~~|~~ I_{u'}=1) \leq \Prob(I_u=1)$, since $\Prob(I_u = 1 ~~|~~ I_{u'}=1) \leq \Prob(I_u = 1 ~~|~~ I_{u'}=0)$ (for a fixed total number of edges, the probability for $u$ to be bad is smaller if another vertex $u'$ is bad) and, by the law of total probability, at least one of the two conditional probabilities has to be at most $\Prob(I_u=1)$. Thus, $\sum_{u \in S} I_u$ is bounded from above by $\sum_{u \in S} I'_u$, where the $I'_u$ are independent indicator random variables with
$$
\Prob(I_u=1)\leq \Prob(I'_u=1)=C\frac{(d-1)^{t+2i}}{n^2}
$$
for some sufficiently large $C>0$. The total number of bad vertices in the subgraph of depth $j$ pending at $v$ is $X=\sum_{i=1}^j \sum_{u \in S} I_u(v,i) \leq \sum_{i=1}^j \sum_{u \in S} I'_u(v,i).$ Since $\frac{(d-1)^{t+2j}}{n} \leq \log n,$
$$
\E(X) \leq \sum_{i=1}^j C\frac{(d-1)^{t+2i}}{n}=O \left(\frac{(d-1)^{t+2j}}{n} \right) = O(\log n),
$$
and by the Chernoff bound given by~(\ref{eq:Ch1}), $X \leq C' \log n$ with probability $1- o(n^{-1})$ for some $C'>0$ large enough.
\end{proof}

\bigskip

We need one more lemma. For a given vertex $v \in L_t$ and integer $j$, a vertex $u \in S$ is called \emph{very bad} if it has at least two incoming edges from vertices in $S(v,i-1)$. In particular, every very bad vertex is bad. Let $Z(v,j)$ be the number of very bad vertices in $N(v,j)$.

For a given $T = \Theta(\log \log n)$, and any $\hat{L}_T \subseteq L_T$ such that $|\hat{L}_T| = o(|L_T| / \log^2 n)$, we will consider the subgraph $G(\hat{L}_T)$ consisting of all vertices to which there is a directed path from some vertex in $\hat{L}_T$. For any $t > T$, let $\hat{L}_t$ be a subset of $L_t$ that is in $G(\hat{L}_T)$.

\begin{lemma}\label{lem:notworoots}
Let $\hat{L}_T \subseteq L_T$ for some $T = \Theta(\log \log n)$ be such that $|\hat{L}_T| = o(|L_T|/\log^2 n)$. Then a.a.s.\ for any $v \in \hat{L}_t$, where $T \le t \leq \log_{d-1}n-\omega$, and any integer $j$ with $\frac{(d-1)^{t+2j}}{n} \leq \log n$, we have that $Z(v,j)=0$.
\end{lemma}
\begin{proof}
Fix any $v \in \hat{L}_t$ for some $T \le t \leq \log_{d-1}n-\omega$. As in Lemma~\ref{lem:badvertices}, by letting $H_u(v,i)$ be the event that $u \in S$ is very bad, we have
$$
\Prob(H_u(v,i)) = \frac{O((d-1)^i)}{n}\frac{O((d-1)^{i})}{n}=O\left(\frac{(d-1)^{2i}}{n^2}\right).
$$
Letting $H_u=H_u(v,i)$ be the corresponding indicator variable, we have that
$$
\E \left(\sum_{u \in S} H_u\right) = O \left(\frac{(d-1)^{2i}}{n} \right).
$$
Analogously as in the previous proof, define independent indicator random variables $H'_u$ with $\Prob(H_u=1)\leq \Prob(H'_u=1)=C\frac{(d-1)^{2i}}{n^2}$. We have
$$
Z(v,j)=\sum_{i=1}^j \sum_{u \in S} H_u(v,i) \leq \sum_{i=1}^j \sum_{u \in S} H'_u(v,i),
$$
and so
$$
\E(Z(v,j)) \leq \sum_{i=1}^j C\frac{(d-1)^{2i}}{n}=O \left(\frac{(d-1)^{2j}}{n} \right) =O \left(\frac{\log n}{(d-1)^t} \right),
$$
since $\frac{(d-1)^{t+2j}}{n} \leq \log n$.

As $|\hat{L}_T| = o(|L_T| / \log^2 n)$, we have that
$$
|\hat{L}_t| \le |\hat{L}_T| (d-1)^{t-T} = o((d-1)^t / \log^2 n),
$$
and so the expected number of very bad vertices found in $\hat{L}_t$ is $o(1/\log n)$. Finally, the expected number of very bad vertices in any sublayer $\hat{L}_t$ ($T \le t \le \log_{d-1} n - \omega$) is $o(1)$, and the result holds by Markov's inequality.
\end{proof}

We now come back to the proof of the main theorem for random regular DAGs.
\begin{proof}[Proof of Theorem~\ref{thm:main_d-reg}(iii) and (iv)]
Note that we already proved an upper bound of (iii) (see the proof of parts (i) and (ii)). Since $g_j$ is non-increasing as a function of $j$, an upper bound of (iv) also holds.

We will prove a lower bound of (iv) first. The lower bound of (iii) will follow easily from there. Let $s \in \N^+$ and suppose that we play the game with parameter $c = d-2-\frac {1}{s}$. If $s \neq 1$, then for every $i \in \N$, we have that $c_{si+1} = d-3$ and $c_t = d-2,$ otherwise. (For $s=1$ we find that $c_t=d-3$ for any $t$.) Suppose that the greens play \emph{greedily} (that is, they always protect vertices adjacent to the sludge) and the graph is locally a tree. Note that during the time between $si+2$ and $s(i+1)$, they can direct the sludge leaving him exactly one vertex to choose from at each time-step. However, at time-step $s(i+1)+1$, the sludge has 2 vertices to choose from. The sludge has to use this opportunity wisely, since arriving at a bad vertex (see definition above) when the greens can protect $d-2$ vertices would result in him losing the game. Our goal is to show that the sludge can avoid bad vertices and, as a result, he has a strategy to reach the sink $L_j$. Since we aim for a statement that holds a.a.s.\ we can assume that all properties mentioned in Lemmas~\ref{lem:layersize},~\ref{lem:badvertices}, and~\ref{lem:notworoots} hold.

Before we describe a winning strategy for the sludge, let us discuss the following useful observation. While it is evident that the greens should use a greedy strategy to play on the tree, it is less evident in our situation. Perhaps instead of playing greedily, the greens should protect a vertex far away from the sludge, provided that there are at least two paths from the sludge to this vertex. However, this implies that the vertex is very bad and we know that very bad vertices are rare. It follows from Lemma~\ref{lem:notworoots} that there is no very bad vertex within distance $j$, provided that the sludge is at a vertex in $L_t$, $t = \Omega(\log \log n)$  and $\frac {(d-1)^{t+2j}}{n} \le \log n$. (For early steps we know that the graph is locally a tree so there are no bad vertices at all.) Therefore, without loss of generality, we can assume that at any time-step $t$ of the game, the greens protect vertices greedily or protect vertices at distance at least $j$ where $j$ is the smallest value such that $\frac {(d-1)^{t+2j}}{n} > \log n$. We call the latter protected vertices \emph{dangerous}. The sludge has to make sure that there are no nearby bad nor dangerous vertices.

Let
$$
T= s (\log_2 \log n + C),
$$
where the constant $C>0$ will be determined soon and is sufficiently large such that the sludge is guaranteed to escape from all bad or dangerous vertices which are close to him. Let $\delta = {3}/{\log_2 \left(\frac {d-1}{d-2} \right)}.$  During the first $\delta T$ time-steps, the sludge chooses any arbitrary branch. Since he is given this opportunity at least $\delta \log_2 \log n = 3 \log_{(d-1)/(d-2)} \log n$  times and each time he cuts the number of possible destinations by a factor of $\frac {d-2}{d-1}$, the number of possible vertices the sludge can reach at time $\delta T$ is $O(|L_{\delta T}| / \log^3 n)$. From that point on, it follows from Lemma~\ref{lem:notworoots} that there are no nearby very bad vertices. At time $t_1= \delta T$, by Lemma~\ref{lem:layersize}, there are no bad vertices at distance
$$
d_1 = \frac12 \log_{d-1}n - \delta T-\omega \ge T
$$
from the sludge, and hence, no dangerous vertices within this distance. It follows from Lemma~\ref{lem:badvertices} that there are $O(\log n)$ bad vertices at distance
$$
\bar{d}_1 = \frac12 \log_{d-1}n + \frac12 \log_{d-1} \log n-\frac{\delta T}{2}.
$$
There are $O(\log n)$ dangerous vertices within this distance (since the total number of protected vertices during the whole game is of this order). Thus, there are $O(\log n)$ bad or dangerous vertices at a distance between $d_1$ and $\bar{d}_1$ from the sludge.

To derive a lower bound on the length of the game, we provide a strategy for the sludge that allows him to play for at least a certain number of steps, independently of the greens' behaviour.  In particular, his goal is to avoid these bad or dangerous vertices: as long as the sludge is occupying a vertex that is not bad, there is at least one vertex on the next layer available to choose from.  More precisely, it follows from Lemma~\ref{lem:notworoots}, that from time $\delta T$ onwards, locally there are no very bad vertices. Let us call a \emph{round} a sequence of $T$ time-steps. Since all bad vertices are in distinct branches,  in every $s$-th time-step the sludge can half the number of bad vertices. Therefore, after one round the sludge can escape from all $(C'+1)\log n$ bad or dangerous vertices that are under consideration in a given round, provided that $C>0$ is large enough constant. (Recall the constant $C'$ is defined in Lemma~\ref{lem:badvertices}.)

Using this strategy, at time $t_2=(\delta + 1)T$ there are no bad or dangerous vertices at distance
$$
d_2 = \bar{d_1} - T = \frac12 \log_{d-1} n + \frac12 \log_{d-1} \log n-\frac {\delta+2}{2} T \ge T.
$$
To see this, note that since the sludge escaped from all bad or dangerous vertices, which at time $t_1$ were at distance $\bar{d_1}$, and he has advanced $T$ steps by now.  Using Lemma~\ref{lem:badvertices} again, we find that there are $O(\log n)$ bad or dangerous vertices at distance
$$
\bar{d_2} = \frac12 \log_{d-1}n + \frac12 \log_{d-1} \log n - \frac{\delta+1}{2} T.
$$
Arguing as before, we find that it takes another $T$ steps to escape from them.

In general, at time $t_i= (\delta+i-1)T$, there are $O(\log n)$ bad or dangerous vertices at a distance between
$$
d_i = \frac12 \log_{d-1} n + \frac12 \log_{d-1} \log n-\frac {\delta+i}{2} T
$$
and
$$
\bar{d}_i = \frac12 \log_{d-1} n + \frac12 \log_{d-1} \log n-\frac {\delta+i-1}{2} T
$$
Thus, as long as $d_i \ge T$, the strategy of escaping from bad or dangerous vertices before actually arriving at that level is feasible. Moreover, we can finish this round, and so the sludge is guaranteed to use this strategy until time $t_i$, where $i$ is the smallest value such that $d_i \le T$. Solving this for $i$ we obtain that
$$
\frac{(\delta+i+2)}{2}T  \leq  \frac12 \log_{d-1}n + \frac12 \log_{d-1} \log n,
$$
and so
$$
t_i = (\delta + i - 1) T \leq  \log_{d-1}n+\log_{d-1}\log n - 3T.$$
Hence,
$$
t_i = (\delta + i - 1) T  \leq  \log_{d-1}n- 3s \log_2 \log n + \log_{d-1} \log n - O(1).
$$
Finally, note that if $i$ is the smallest value such that $d_i \leq T$, we get that $d_{i-1}\geq T$ and so $d_i \geq \frac{T}{2}$. Hence, another $T/2$ steps can be played, and the constant of the second order term can be improved from $3s \log_2 \log n$ to $\frac52 s \log_2 \log n$, yielding part (iv). Part (iii) follows by taking $s$ to be a function of $n$ slowly growing to infinity.
\end{proof}

\bigskip

Our next goal is to show that when $j = \log_{d-1} n + C$, the value of $g_j$ is slightly smaller than $d-2$, provided that $C$ is a sufficiently large constant. However, before we do it, we need one more observation. It follows from Lemma~\ref{lem:layersize}(i) that a.a.s.\ $|L_t| = (1-o(1)) d(d-1)^{t-1}$ for $t = \log_{d-1} n - \omega$ ($\omega = \omega(n)$ is any function tending to infinity with $n$, as  usual). However, this is not the case when $t = \log_{d-1} n + O(1)$. At this point of the process, a positive fraction of vertices of $L_t$ are bad. This, of course, affects the number of edges from $L_t$ to $L_{t+1}$. In fact, the number of edges between two consecutive layers converges to $c_0 n$ as shown in the next lemma.

\begin{lemma}\label{lem:edges_conv}
Let $c_0$ be the constant satisfying
$$
\sum_{k=1}^{d-1} (d-k)\frac{c^k}{k!}e^{-c}=c.
$$
For every $\eps >0$, there exists a constant $C_{\eps}$ such that a.a.s.\ for every $\log_{d-1} n + C_{\eps} \le t \le 2 \log_{d-1} n$,
$$
(1-e^{-c_0+\eps})n ~~\le~~ |L_t| ~~\le~~ (1-e^{-c_0-\eps})n,
$$
and the number of edges between $L_t$ and $L_{t+1}$ is at least $(c_0-\eps)n$ and at most $(c_0+\eps)n$.
\end{lemma}

\begin{proof}
Suppose that the layer $L_t$ has in total $cn$ random incoming edges, for some $c=c(n) \in (0,1]$. Then the probability that a vertex $v \in S$ (recall that $S$ is the set of cardinality $n$ used to create layer $L_t$) has in-degree $k \in \N $ (that is, absorbs $k$ incoming edges, or attracts no edge if $k=0$) is
$$
\binom{cn}{k} \left(\frac{1}{n}\right)^k \left( 1-\frac{1}{n} \right)^{cn-k} = (1+o(1)) \frac{c^k}{k!}e^{-c}.
$$
Note that each vertex of in-degree $1 \le k \le d-1$ generates $d-k$ edges to the next layer. Further, vertices of in-degrees $k$ or more do not have any offspring, and vertices of $S$ of in-degree zero are not in $L_t$. Therefore, the expected number of outgoing edges produced by all vertices in layer $L_t$ is
$$
(1+o(1)) \sum_{k=1}^{d-1} (d-k) \frac{c^k}{k!}e^{-c}n.
$$
The events considered here are almost independent (one can compute higher moments and see that the $k$th moment is asymptotically equal to the $k$th power of the first moment), so for any $0 \leq k \leq d-1$ it follows from Chernoff bounds that with probability $1-o(\log^{-1} n)$ the number of vertices of degree $k$ is $(1+o(1)) c^k e^{-c}/k!$. Thus, with the same probability, strong concentration also follows for the number of edges. If the number of incoming edges equals $c_0 n$, then the expected number of outgoing edges equals
$$
(1+o(1)) \sum_{k=1}^{d-1} (d-k)\frac{c_0^k}{k!}e^{-c_0}n = (1+o(1)) c_0 n.
$$
If less than $c_0 n$ edges are incoming, then more will be going out, and vice versa. A.a.s.\ the process converges and so there exists a constant $C_{\eps}$ such that a.a.s.\ the number of edges between two consecutive layers $L_t$ and $L_{t+1}$ is between $(c_0-\eps)n$ and $(c_0+\eps)n$ for any $t$ such that $\log_{d-1} n + C_{\eps} \le t \le 2 \log_{d-1} n$.

Finally, let us recall that the layer $L_t$ consists of vertices of $S$ with in-degree at least one. The number of in-degree $0$ vertices is concentrated around its expectation, and thus we have that a.a.s.\
$$
(1-e^{-c_0+\eps})n \le |L_t| \le (1-e^{-c_0-\eps})n.
$$
The lemma is proved.
\end{proof}

The value of $c_0$ (and so $1-e^{-c_0}$ as well) can be numerically approximated. It is straightforward to see that $c_0$ tends to $d/2$ (hence, $1-e^{-c_0}$ tends to $1$) when $d \to \infty$. We present below a few approximate values.

\begin{table}[htdp]
\begin{center}
\begin{tabular}{|c|c|c|c|c|c|}
\hline
$d$ & 3 & 4 & 5 & 10 & 20 \\
\hline
$c_0$ & 0.895 & 1.62 & 2.26 & 4.98 & $\approx 10$ \\
$1-e^{-c_0}$ & 0.591 & 0.802 & 0.895 & 0.993 & $\approx 1$ \\
\hline
\end{tabular}
\end{center}
\caption{Approximate values of $c_0$ and $1-e^{-c_0}$.}\label{tab:small_values}
\end{table}

Finally, we are ready to finish the last part of Theorem~\ref{thm:main_d-reg}.\smallskip

\noindent \emph{Proof of Theorem~\ref{thm:main_d-reg}(v)}.
We assume that the game is played with parameter $c = d-2-\frac {1}{s}$ for some $s \in \N^+ \setminus \{1,2,3\}$. For every $i \in \N$, we have that $c_{si+1} = d-3$, and $c_t = d-2$, otherwise. To derive an upper bound of $g_j$ that holds a.a.s., we need to prove that a.a.s.\ there exists no winning strategy for the sludge.

We will use a combinatorial game-type argument. The greens will play \emph{greedily} (that is, they will always protect nodes adjacent to the sludge). Suppose that the sludge occupies node $v \in L_{si+1}$ for some $i \in \N$ (at time $t=si+2$ he moves from $v$ to some node in $L_t$) and he has a strategy to win from this node, provided that no node in the next layers is protected by the greens. We will call such a node \emph{sludge-win}. Note that during the time period between $si+2$ and $s(i+1)$, the greens can protect $d-2$ nodes at a time, so they can direct the sludge leaving him exactly one node to choose from at each time-step. Therefore, if there is a node of in-degree at least 2 in any of these layers, the greens can force the sludge to go there and finish the game in the next time-step. This implies that all nodes within distance $s-2$ from $v$ (including $v$ itself) must have in-degree 1 and so the graph is locally a tree. However, at time-step $s(i+1)+1$, the greens can protect $d-3$ nodes, one less than in earlier steps. If the in-degree of a node reached at this layer is at least 3, then the greens can protect all out-neighbours and win. Further, if the in-degree is 2 and there is at least one out-neighbour that is not sludge-win, the greens can force the sludge to go there and win by definition of not being sludge-win.  Finally, if the in-degree is 1, the sludge will be given 2 nodes to choose from. However, if there are at least two out-neighbours that are not sludge-win, the greens can ``present'' them to the sludge and regardless of the choice made by the sludge, the greens win.

We summarize now the implications of the fact that $v \in L_{si+1}$ is sludge-win. First of all, all nodes within distance $s-2$ are of in-degree 1. Nodes at the layer $L_{s(i+1)}$ below $v$ have in-degree at most 2. If $u \in L_{s(i+1)}$ has in-degree 2, then all of the $d-2$ out-neighbours are sludge-win. If $u \in L_{s(i+1)}$ has in-degree 1, then all out-neighbours except perhaps one node are sludge-win. Using this observation, we characterize a necessary condition for a node $v \in L_1$ to be sludge-win. For a given $v \in L_1$ that can be reached at time 1, we define a \emph{sludge-cut} to be the following cut: examine each node of $L_{si}$, and proceed inductively for $i \in \N^+$. If $u \in L_{si}$ has out-degree $d-1$, then we cut away any out-neighbour and all nodes that are not reachable from $v$ (after the out-neighbour is removed). The node that is cut away is called an \emph{avoided node}. After the whole layer $L_{si}$ is examined, we skip $s-1$ layers and move to the layer $L_{s(i+1)}$. We continue until we reach the sink, the layer $L_j=L_{si'}$ for some $i'$ (we stop at $L_j$ without cutting any further). The main observation is that if the sludge can win the game, then the following claim holds.

\bigskip

\noindent \emph{Claim}. There exists a node $v \in L_1$ and a sludge-cut such that the graph left after cutting is a $(d-1,d-2)$-regular graph, where each node at layer $L_{si}$, $1 \le i \le i'-1$ has out-degree $d-2$, and all other nodes have out-degree $d-1$. In particular, for any $1 \le i \le i'-1$ the graph induced by the set $\bigcup_{t=si}^{s(i+1)-1} L_t$ is a tree.

\bigskip

It remains to show that a.a.s.\ the claim does not hold.  (Since there are at most $d$ nodes in $L_1$ it is enough to show that a.a.s.\ the claim does not hold for a given node in $L_1$.) Fix $v \in L_1$. The number of avoided nodes at layer $L_{si+1}$ is at most the number of nodes in $L_{si}$ (after cutting earlier layers), which is at most
$$
\ell_i = (d-1)^{si-1} \left(\frac{d-2}{d-1} \right)^{i-1} = (d-1)^{(s-1)i}(d-2)^{i-1}.
$$
In particular, $\ell$, the number of nodes in the sink after cutting, is at most $\ell_{i'} \le n$. It can be shown that a.a.s.\ $\ell > n^{\alpha}$ for some $\alpha > 0$.

Fix $n^{\alpha} \le \ell \le \ell_{i'}\le n$. We need to show that for this given $\ell$ the claim does not hold with probability $1-o(n^{-1})$. Since each node in $L_{si'}$ has in-degree at most 2, the number of nodes in $L_{si'-1}$ is at most $2\ell$ (as before, after cutting). Since the graph between layer $L_{s(i'-1)}$ and $L_{si'-1}$ is a tree, the number of nodes in $L_{si'}$ is at most $2 \ell/(d-1)^{s-1}$, which is an upper bound for the number of avoided nodes at the next layer $L_{si'+1}$. Applying this observation recursively we obtain that the total number of avoided nodes up to layer $si'$ is at most
$
4 (d-1)^{-s+1} \ell .
$
To count the total number of sludge-cuts of a given graph, observe that each avoided node corresponds to one out of $d-1$ choices. Hence, the total number of sludge-cuts is at most \begin{equation}\label{eq:tree-cuts}
(d-1)^{4 (d-1)^{-s+1} \ell}.
\end{equation}

We now estimate the probability that the claim holds for a given $v \in L_1$ and a sludge-cut. To obtain an upper bound, we estimate the probability that all nodes in the layer $L_{si'-1}$ are of in-degree 1. Conditioning on the fact that we have $\ell$ nodes in the last layer, we find that the number of nodes in $L_{si'-1}$ is at least $\frac {\ell}{d-1}.$ Let $i'$ be large enough such that we are guaranteed by Lemma~\ref{lem:edges_conv} that the number of edges between the two consecutive layers is at least $c_0n(1-\eps/2)$. Hence, the probability that a node in $L_{si'-1}$ has in-degree 1 is at most
\begin{equation}\label{woo}
\left(1 - \frac {1}{n} \right)^{c_0n(1-\eps/2)} = (1+o(1)) e^{-c_0(1-\eps/2)}  \le e^{-c_0(1-\eps)} ,
\end{equation}
where $\eps>0$ can be arbitrarily small by taking $i'$ large enough. Let $p_{\eps}$ be the probability in (\ref{woo}). We derive that $j=si' \ge \log_{d-1} n + C'$, where $C'=C'(\eps,s)>0$ is a large enough constant. Conditioning under $v \in L_{si'-1}$ having in-degree $1$, it is harder for $v' \in L_{si'-1}$ to have in-degree $1$ than without this condition, as more edges remain to be distributed. Thus, the probability that all nodes in $L_{si'-1}$ have the desired in-degree is at most
\begin{equation}\label{eq:in-degree1}
p_\eps^{\frac {\ell}{d-1}} = \exp \left( -c_0 (1-\eps) \frac {\ell}{d-1} \right).
\end{equation}
Thus, by taking a union bound over all possible sludge-cuts (the upper bound for the number of them is given by~\eqref{eq:tree-cuts}), the probability that the claim holds is at most
$$
\left( (d-1)^{4 (d-1)^{-s+1}}
\left( e^{-c_0(1-\eps)} \right)^{\frac {1}{d-1}} \right)^{\ell}
$$
which can be made $o(n^{-1})$ by taking $\eps$ small enough, provided that $s$ is large enough so that
$$
(d-1)^{4 (d-1)^{-s+2} }
e^{-c_0} < 1.
$$
By considering the extreme case for the probability of having in-degree one when $d=3$ we obtain that
$$
e^{-c_0} \le e^{- \frac {0.895}{3}d} \le e^{-0.29d}
$$
for $d \ge 3$ (see Table~\ref{tab:small_values}). It is straightforward to see that $s \ge 4$ will work for any $d \ge 3$, and $s \ge 3$ for $d \ge 5$. \hfill $\Box$ \smallskip

\section{Proofs for random power law DAGs}

Let us recall that we have three parameters in this model: $\beta>2$, $d>0$, and $0 < \alpha < 1$. For a given set of parameters, we defined
$$
M=M(n)=n^{\alpha}, \mbox{\ \ \ }
i_0=i_0(n)=n \left(\frac{d}{M} \frac{\beta-2}{\beta-1}\right)^{\beta-1}, \mbox{\ and\ \ \ }
c=\left(\frac{\beta-2}{\beta-1}\right) dn^{\frac{1}{\beta-1}}.
$$
Finally, for $i \ge 1$ we have that
$$
w_i = c (i_0+i-1)^{-\frac{1}{\beta-1}}.
$$

Before we analyze the game for this model, let us focus on investigating some properties of the random graph we play on. We already mentioned that the sequence $(w_i)_{i \in \N}$ is decreasing but it is not obvious which weights we obtain for a given level $L_j$. We start by providing a lower bound for the weight of vertices in each layer $j$ (which will imply an upper bound for the previous layer $j-1$). Since the weight $w_i$ is a function of the index $i$, it is enough to focus on the latter. For $j \in \N$, let $\ell_j$ be the smallest index among the vertices of layer $L_j$. Using the notation introduced in Section~\ref{sec:definitions}, $\ell_j = d_{j-1}+1$. The maximum weight at $L_j$ is $w_{\ell_j}$, the minimum one is $w_{\ell_{j+1}-1}$.  The first lemma investigates the behaviour of $\ell_j$.

\begin{lemma}\label{lem:max_deg}
Let
$$
\gamma = d^{\beta-1} \left( \frac {\beta-2}{\beta-1} \right)^{\beta-2} \left( \left(1+ \left( d \frac{\beta-2}{\beta-1} \right)^{1-\beta} \right)^{\frac{\beta-2}{\beta-1}}-1\right)
$$
if $\frac {1}{\alpha} - \beta + 3 \in \N^+ \setminus \{1,2\}$, and $\gamma = 1,$ otherwise. Let
$$
\xi = \left(\frac {\beta-2}{\beta-1}\right) d \left( \left( \frac {d(\beta-2)}{\beta-1} \right)^{\beta-1} + 1 \right)^{-\frac {1}{\beta-1}}.
$$

The following holds a.a.s.
\begin{itemize}
\item [(i)] $\ell_0 = 1$, $\ell_1 = 2$, and $\ell_2 = (1+o(1))M$. \\
In particular, $w_{\ell_0} = M$, $w_{\ell_1} = (1+o(1)) M$, and
$$
w_{\ell_2} = (1+o(1))
\begin{cases}
n^{\alpha} & \text{ if } \alpha < \frac {1}{\beta} \\
\xi n^{\alpha} & \text{ if } \alpha = \frac {1}{\beta} \\
\left( \frac{\beta-2}{\beta-1} \right) d n^{\frac {1-\alpha}{\beta -1}} & \text{ if } \alpha > \frac {1}{\beta}.
\end{cases}
$$
\item [(ii)] For $3 \le j < \frac {1}{\alpha} - \beta + 3$ we have that
$$\ell_j = (1+o(1)) M^{j-1} = (1+o(1)) n^{\alpha(j-1)}.$$
In particular,
$$
w_{\ell_j} = (1+o(1))
\begin{cases}
n^{\alpha} & \text{ if } 3 \le j < \frac {1}{\alpha} - \beta + 2 \\
\xi n^{\alpha} & \text{ if } 3 \le j = \frac {1}{\alpha} - \beta + 2 \\
\left( \frac{\beta-2}{\beta-1} \right) d n^{\frac {1-\alpha(j-1)}{\beta -1}} & \text{ if } \frac {1}{\alpha} - \beta + 2 < j < \frac {1}{\alpha} - \beta + 3 \text{ and } j \ge 3.
\end{cases}
$$
\item [(iii)] If $j_0 = \frac {1}{\alpha} - \beta + 3 \in \N^+ \setminus \{1,2\}$, then,
$$\ell_{j_0} = (1+o(1)) \gamma M^{j_0-1} = (1+o(1)) \gamma n^{\alpha(j_0-1)} = \Theta(n^{\alpha(j_0-1)}).$$
In particular,
$$
w_{\ell_{j_0}} = (1+o(1)) \gamma^{-\frac {1}{\beta-1}} \left( \frac{\beta-2}{\beta-1} \right) d n^{\frac {1-\alpha(j_0-1)}{\beta -1}}.
$$
\item [(iv)] Let $j_1$ be the largest integer satisfying $j_1 \le \max\{ \frac {1}{\alpha} - \beta + 3, 2\}$. Let $j_2 = O(\log \log n)$ be the largest integer such that
$$
d^{\beta-1} \left(\frac{\gamma}{d^{\beta-1}} n^{\alpha(j_1-1)-1} \right)^{\left(\frac{\beta-2}{\beta-1}\right)^{j_2-j_1}} \le (\omega \log \log n)^{- \max\{ 2, \beta-2\}}.
$$
Then for $j_1 < j \le j_2$ we have that
$$
\ell_j = (1+o(1)) d^{\beta-1}n \left(\frac{\gamma}{d^{\beta-1}} n^{\alpha(j_1-1)-1} \right)^{\left(\frac{\beta-2}{\beta-1}\right)^{j-j_1}}.
$$
In particular, for $j_1 < j \le j_2$ we have that
$$
w_{\ell_j} = (1+o(1)) \left(\frac {\beta-2}{\beta-1} \right) \left(\frac{\gamma}{d^{\beta-1}} n^{\alpha(j_1-1)-1} \right)^{-\left(\frac{\beta-2}{\beta-1}\right)^{j-j_1}/(\beta-1)}.
$$
\item [(v)] For any $1 \leq j < j_2$, the number of edges between $L_{j-1}$ and $L_{j}$ is $(1+o(1))\ell_{j+1}$.�
\end{itemize}
\end{lemma}
\begin{proof}
Clearly, we have $\ell_0=1$, $\ell_1=2$; (i) holds deterministically for $j=0,1$.  The number of vertices on levels $0$ and $1$ is at most $1+w_1=1+M$ but can be slightly smaller if there are some parallel edges (which happens a.a.s.\ if $\alpha>1/2$). We derive a deterministic upper bound for $\ell_2$ of $2+M$ but in fact, using the first moment method, we can show that a.a.s. $\ell_2=(1+o(1))M$. Indeed, the probability that a given vertex from $S$ has in-degree at least 2 is $(1+o(1)) {M \choose 2}/n^2$ so we expect $O(M^2/n)$ vertices of in-degree at least 2. With probability $1-O(1/(\omega \log \log n))$ we have $O(M^2 \omega \log \log n / n)$ of such vertices and so $\ell_2 \ge M (1-O(M \omega \log \log n / n))$. The statement for $j=2$ holds, and hence, part (i) follows.

Now, let us generalize this observation. Let $j \ge 3$ and suppose that $\ell_{j-1}$ is already estimated. Note that $\ell_j = \ell_{j-1} + |L_{j-1}|$ so it remains to estimate the size of $L_{j-1}$. We obtain that
\begin{eqnarray*}
|L_{j-1}| \le \ellb_{j-1} = \sum_{i=\ell_{j-2}}^{\ell_{j-1}-1} (w_i-1) &=& O(w_{\ell_{j-2}})+ O(\ell_{j-1})+ \int_{i=\ell_{j-2}}^{\ell_{j-1}} c(i_0+i-1)^{-\frac{1}{\beta-1}} di \\
&=& O(w_{\ell_{j-2}})+ O(\ell_{j-1}) + \int_{i=1}^{\ell_{j-1}} c(i_0+i-1)^{-\frac{1}{\beta-1}} di \\
&=& O(M)+ O(\ell_{j-1}) + c\left(\frac{\beta-1}{\beta-2}\right)\left((\ell_{j-1}+i_0)^{\frac{\beta-2} {\beta-1}}-i_0^{\frac{\beta-2}{\beta-1}} \right) \\
&=& O(\ell_{j-1}) +c\left(\frac{\beta-1}{\beta-2}\right)i_0^{\frac{\beta-2}{\beta-1}}\left( \left(1+\frac{\ell_{j-1}}{i_0}\right)^{\frac{\beta-2}{\beta-1}}-1\right).
\end{eqnarray*}
Note that  $\ellb_{j-1}$ is an upper bound for the number of edges between layer $L_{j-2}$ and $L_{j-1}$ (and so an upper bound for $|L_{j-1}|$), and we derive the equality $|L_{j-1}| = \ellb_{j-1}$ if all vertices in $L_{j-1}$ and $L_{j-2}$ have in-degree~1. Arguing as before, we deduce that with probability $1-O(1/(\omega \log \log n))$ the number of edges going to vertices in $L_{j-2}$ (in $L_{j-1}$) that are of degree at least 2 is $O(\ell_{j-1}^2 \omega \log \log n / n)$ ($O(\ellb_{j-1}^2 \omega \log \log n / n)$, respectively). Each edge of this type directed to a vertex in $L_{j-2}$ affects its out-degree, and so decreases the number of vertices in $L_{j-1}$ by at most one. Similarly, one edge going to a vertex in $L_{j-1}$ of in-degree at least 2 decreases by at most one the number of vertices of in-degree 1. Thus, with probability $1-O(1/(\omega \log \log n))$, by considering vertices of in-degree 1 only, we obtain that
\begin{eqnarray}
|L_{j-1}| &\ge& \ellb_{j-1} - O(\ellb_{j-1}^2 \omega \log \log n / n) - O(\ell_{j-1}^2 \omega \log \log n / n) \nonumber \\
&=& \ellb_{j-1} - O(\ellb_{j-1}^2 \omega \log \log n / n).\label{eq:num_edges}
\end{eqnarray}
(The last equality follows from the fact that $\ellb_{j-1} = \Omega(\ell_{j-1})$, provided that $w_{\ell_{j-1}} = \Omega(1)$. In fact, we consider values of $j$ at most $j_2$ for which it will be shown that $w_{\ell_{j-1}} \ge \omega$ and so $\ellb_{j-1} >  \ell_{j-1}$.) This, together with the fact that $\ell_j = \ell_{j-1} + |L_{j-1}|$, implies that
\begin{equation}\label{ar:equations}
\ell_j = O(\ell_{j-1})-O(\ellb_{j-1}^2 \omega \log \log n / n)+c\left(\frac{\beta-1}{\beta-2}\right)i_0^{\frac{\beta-2}{\beta-1}}\left( \left(1+\frac{\ell_{j-1}}{i_0}\right)^{\frac{\beta-2}{\beta-1}}-1\right).
\end{equation}

If $\ell_{j-1}=o(i_0)$, then
\begin{eqnarray*}
\ell_j &=& O(\ell_{j-1})+O(\ellb_{j-1}^2 \omega \log \log n / n) +c\left(\frac{\beta-1}{\beta-2}\right)i_0^{\frac{\beta-2}{\beta-1}}\left( \frac{\beta-2}{\beta-1}\frac{\ell_{j-1}}{i_0}+O \left(\frac{\ell_{j-1}}{i_0} \right)^2\right) \\
&=& O(\ell_{j-1})+O(\ellb_{j-1}^2 \omega \log \log n / n)+M \ell_{j-1} \left(1+O \left(\frac{\ell_{j-1}}{i_0}\right) \right) \\
&=& M\ell_{j-1} \left(1+O \left(\frac{\ell_{j-1}}{i_0}\right) \right) \left(1+O \left(M^{-1} \right) \right) \left(1+O \left(\frac {M \ell_{j-1} \omega \log \log n}{n} \right) \right).
\end{eqnarray*}
Note that $M^{j-2} = n^{(j-2)\alpha}$ and $i_0 = \Theta(n^{1-\alpha(\beta-1)})$. Therefore this recursive formula is to be applied $O(1)$ times only before the condition $\ell_{j-1} = o(i_0)$ fails (it may, of course, happen that it fails for $j=3$ so we do not apply it at all). We have that a.a.s.\ the statement holds for any value of $j$ such that $(j-2) \alpha < 1-\alpha(\beta-1)$; that is, $j < \frac {1}{\alpha} - \beta + 3.$ Moreover, the error term can be estimated much better; it is, in fact, $(1+O(n^{-\eps}))$ for some $\eps>0$. Let us note one more time that it may happen that  $\frac {1}{\alpha} - \beta \le 0$ and so the condition fails for $j=3$ but then (ii) trivially holds. Thus, part (ii) is finished.

For part (iii), suppose that $j_0 = \frac {1}{\alpha} - \beta + 3 \in \N^+ \setminus \{1,2\}$. Since our goal is to show that the statement holds a.a.s., we may assume that $\ell_{j_0-1} = (1+O(n^{-\eps})) M^{j_0-2}$ for some $\eps>0$. From the assumption it follows that $\ell_{j_0-1}$ and $i_0$ are of the same order. By the relations between $i_0$ and $M$, we have $M^{\beta-1}= n(d\frac{\beta-2}{\beta-1})^{\beta-1}/ i_0$. Thus, $\ell_{j_0-1}=(1+O(n^{-\eps})) M^{-(\beta-1)+\frac{1}{\alpha}}$, and hence, $\ell_{j_0-1} = (1+O(n^{-\eps})) \left( d \frac {\beta-2}{\beta-1} \right)^{1-\beta} i_0$.  It follows from~(\ref{ar:equations}) that a.a.s.
\begin{eqnarray*}
\ell_{j_0} &=& O(\ell_{j_0-1})+O(\ellb_{j_0-1}^2 \omega \log \log n / n)+c\left(\frac{\beta-1}{\beta-2}\right)i_0^{\frac{\beta-2}{\beta-1}}\left( \left(1+\frac{\ell_{j_0-1}}{i_0}\right)^{\frac{\beta-2}{\beta-1}}-1\right) \\
&=& (1+O(n^{-\eps})) c \frac{\beta-1}{\beta-2}i_0^{\frac{\beta-2}{\beta-1}}\left( \left(1+ \left( d \frac{\beta-2}{\beta-1} \right)^{1-\beta} \right)^{\frac{\beta-2}{\beta-1}}-1\right) \\
&=& (1+O(n^{-\eps})) \ell_{j_0-1} M d^{\beta-1} \left( \frac {\beta-2}{\beta-1} \right)^{\beta-2} \left( \left(1+ \left( d \frac{\beta-2}{\beta-1} \right)^{1-\beta} \right)^{\frac{\beta-2}{\beta-1}}-1\right),
\end{eqnarray*}
so (iii) holds.

For part (iv), let $j_1$ be the largest integer satisfying $j_1 \le \max\{ \frac {1}{\alpha} - \beta + 3, 2\}$. Based on earlier parts, we may assume that $\ell_{j_1} = (1+O(n^{-\eps})) \gamma M^{j_1-1}$. Note that $\ell_{j_1} / i_0 = \Omega(n^{\eps})$ for some $\eps >0$, and so $\ell_{j-1}  / i_0 = \Omega(n^{\eps})$ for any $j_1 < j \leq j_2$, since $\ell_j$ is monotonic as a function of $j$.

Fix $j > j_1$. This time we derive from~(\ref{ar:equations}) that with probability $1-O(1/(\omega \log \log n))$\begin{eqnarray*}
\ell_{j} &=& O(\ell_{j-1})+O(\ellb_{j-1}^2 \omega \log \log n / n)+c\left(\frac{\beta-1}{\beta-2}\right)i_0^{\frac{\beta-2}{\beta-1}}\left( \left(1+\frac{\ell_{j-1}}{i_0}\right)^{\frac{\beta-2}{\beta-1}}-1\right) \\
&=& O(\ell_{j-1})+O(\ellb_{j-1}^2 \omega \log \log n / n)+ c\left(\frac{\beta-1}{\beta-2}\right) i_0^{\frac{\beta-2}{\beta-1}} \left(\frac{\ell_{j-1}}{i_0}\right)^{\frac{\beta-2}{\beta-1}} (1+O(n^{-\eps})) \\
&=& O(\ellb_{j-1}^2 \omega \log \log n / n)+ c\left(\frac{\beta-1}{\beta-2}\right) \ell_{j-1}^{\frac{\beta-2}{\beta-1}} (1+O(n^{-\eps})) (1+O(c^{-1} \ell_{j-1}^{\frac{1}{\beta-1}})) \\
&=& O(\ellb_{j-1}^2 \omega \log \log n / n)+ c\left(\frac{\beta-1}{\beta-2}\right) \ell_{j-1}^{\frac{\beta-2}{\beta-1}} (1+O(n^{-\eps})) (1+O((\ell_{j-1}/n)^{\frac{1}{\beta-1}})) \\
&=& (1+O(1/(\omega \log \log n))) c\left(\frac{\beta-1}{\beta-2}\right) \ell_{j-1}^{\frac{\beta-2}{\beta-1}},
\end{eqnarray*}
provided that
\begin{eqnarray*}
c \left(\frac {\beta-1}{\beta-2}\right) \ell_{j-1}^{\frac{\beta-2}{\beta-1}} (\omega \log \log n) / n &\le& (\omega \log \log n)^{-1}, \mbox{\ \ \and}\\
\ell_{j-1}/n &\le& \left( \frac {2}{d} \right)^{\frac {\beta-1}{\beta-2}} (\omega \log \log n)^{-(\beta-1)}.
\end{eqnarray*}
(Note that we have $\ellb_{j-1}=O(c \ell_{j-1}^{\frac{\beta-2}{\beta-1}})$, and thus we obtain the first condition, coming from the term $O(\ellb_{j-1}^2 \omega \log \log n / n)$.) The first condition is equivalent to $\ell_{j-1}/n \le \left( \frac {2}{d} \right)^{\frac {\beta-1}{\beta-2}} (\omega \log \log n)^{-2 \frac{\beta-1}{\beta-2}}$, and so both conditions combined together are equivalent to
$$
\ell_{j-1}/n \le \left( \frac {2}{d} \right)^{\frac {\beta-1}{\beta-2}} (\omega \log \log n)^{- \max\{ 2 \frac{\beta-1}{\beta-2}, \beta -1\}}.
$$
If this condition is satisfied, then we obtain that $\ell_j = (1+o(1)) n^{\frac{1}{\beta-1}}\ell_{j-1}^{\frac {\beta-2}{\beta-1}} d$, or equivalently $\ell_j/n =  (1+o(1)) (\ell_{j-1}/n)^{\frac {\beta-2}{\beta-1}}d$. By using the condition on $\ell_{j-1}/n$, we obtain the following slightly stronger condition (where we ignore the factor of $2$) for $\ell_j$:
\begin{equation}\label{eq:cond}
\ell_{j}/n \le (\omega \log \log n)^{- \max\{ 2, \beta-2\}}.
\end{equation}

Now, suppose that (\ref{eq:cond}) is satisfied, and we rewrite the relation between $\ell_j$ and $\ell_{j-1}$ using the fact that $c=\frac{\beta-2}{\beta-1}dn^{\frac{1}{\beta-1}}$:
$$
\frac{\ell_j }{d^{\beta-1}n} = \left( (1+O(1/(\omega \log \log n))) \frac{\ell_{j-1}}{d^{\beta-1}n} \right)^{\frac{\beta-2}{\beta-1}} =  (1+O(1/(\omega \log \log n))) \left( \frac{\ell_{j-1}}{d^{\beta-1}n} \right)^{\frac{\beta-2}{\beta-1}}.
$$
Applying this argument recursively, we obtain that
$$
\ell_j = d^{\beta-1}n (1+O(1/(\omega \log \log n)))^j \left(\frac{\ell_{j_1}}{d^{\beta-1}n} \right)^{\left(\frac{\beta-2}{\beta-1}\right)^{j-j_1}}.
$$
Finally, since we will soon show that $j \le j_2 = O(\log \log n)$ we derive by the previous cases that
$$
\ell_j = (1+o(1)) d^{\beta-1}n \left(\frac{\gamma M^{j_1-1}}{d^{\beta-1}n} \right)^{\left(\frac{\beta-2}{\beta-1}\right)^{j-j_1}} = (1+o(1)) d^{\beta-1}n \left(\frac{\gamma}{d^{\beta-1}} n^{\alpha(j_1-1)-1} \right)^{\left(\frac{\beta-2}{\beta-1}\right)^{j-j_1}}.
$$
Indeed, since
$$
\ell_j = d^{\beta-1} n \exp \left(  (1+o(1)) \left(\frac{\beta-2}{\beta-1}\right)^{j-j_1} (\alpha(j_1-1)-1) \log n \right),
$$
the condition (\ref{eq:cond}) fails for $j = C \log \log n$ (by taking $C>0$ large enough), and item (iv) follows.

Finally, the proof of part (v) follows now by inspecting closely the parts (ii), (iii), and (iv). In each case, $\ell_{j+1}-\ell_j$ is the size of $L_j$ and it follows from earlier parts that $\ell_{j+1}-\ell_j = (1+o(1)) \ell_{j+1}$. This is clearly a lower bound for the number of edges we try to estimate. On the other hand, $\ellb_j$ serves as an upper bound. Hence, it remains to show that $\ellb_j = (1+o(1)) \ell_{j+1}$.
If $j+1 < \frac{1}{\alpha}-\beta+3$, then $\ell_j=o(i_0)$. By looking at part (ii), we see that the leading term of both $\ellb_{j}$ and $\ell_{j+1}$ is $(1+o(1))c\frac{\beta-2}{\beta-1}i_0^{\frac{\beta-2}{\beta-1}}\left(\frac{\beta-2}{\beta-1}\frac{\ell_{j}}{i_0}+O(\frac{\ell_j}{i_0})^2\right)$, and the result follows for this case (alternatively, in this case we can also observe $\ell_j=(1+o(1))n^{\alpha(j-1)}$, $\ell_{j+1}=(1+o(1))n^{\alpha j}$, and by the trivial bound on the degree, $\ellb_j \leq \ell_j M = (1+o(1))n^{\alpha j}$ holds).
 Next, if $j+1 = \frac{1}{\alpha}-\beta+3$, then $\ell_j=\Theta(i_0)$, and by the calculations of part (iii), in both $\ell_{j+1}$ and in $\ellb_j$ the leading term is of order $(1+o(1))c\frac{\beta-1}{\beta-2}i_0^{\frac{\beta-2}{\beta-1}}\left( \left(1+\frac{\ell_{j}}{i_0}\right)^{\frac{\beta-2}{\beta-1}}-1\right)$, and the result follows also for this case. Finally, if $j+1 > \frac{1}{\alpha}-\beta+3$, then $\ell_j=\omega(i_0)$, $\ellb_{j}=(1+o(1))c\frac{\beta-1}{\beta-2}\ell_{j}^{\frac{\beta-2}{\beta-1}}$, and as observed in part (iv), $\ell_{j+1}=(1+o(1))n^{\frac{1}{\beta-1}}\ell_{j}^{\frac{\beta-2}{\beta-1}}d$, and thus $\ellb_{j}=(1+o(1))\ell_{j+1}$, and part (v) follows.
\end{proof}

We are now ready to come back to investigating the green number. We provide some obvious bounds for the green number and after that we sketch the idea that could be used to estimate it precisely. However, we do not perform these calculations rigorously, since the approach is rather delicate.

\begin{lemma}\label{lem:green_power_law}
Let $\gamma$ and $j_1$ be defined as in Lemma~\ref{lem:max_deg}. That is, let
$$
\gamma = d^{\beta-1} \left( \frac {\beta-2}{\beta-1} \right)^{\beta-2} \left( \left(1+ \left( d \frac{\beta-2}{\beta-1} \right)^{1-\beta} \right)^{\frac{\beta-2}{\beta-1}}-1\right)
$$
if $\frac {1}{\alpha} - \beta + 3 \in \N^+ \setminus \{1,2\}$, and $\gamma = 1$ otherwise. Let $j_1$ be the largest integer satisfying $j_1 \le \max\{ \frac {1}{\alpha} - \beta + 3,2\}$. Moreover, let $j_2 = O(\log \log n)$ be the largest integer such that
$$
d^{\beta-1} \left(\frac{\gamma}{d^{\beta-1}} n^{\alpha(j_1-1)-1} \right)^{\left(\frac{\beta-2}{\beta-1}\right)^{j_2-j_1}} \le (\omega \log \log n)^{- \max\{ 2, (\beta-1)^2\}}.
$$
Then, for $1 \leq  j \le j_2-1$ we have that a.a.s.
$$
(1-o(1)) w_{\ell_{j}} \le g_j \le w_{\ell_{j-1}}.
$$
\end{lemma}
Note that the definition of $j_2$ in Lemma~\ref{lem:green_power_law} is slightly modified compared to the one from Lemma~\ref{lem:max_deg}. However there is only a $O(1)$ difference between these values; in this case it is smaller.

\begin{proof}[Proof of Lemma~\ref{lem:green_power_law}]
Fix $1 \leq j \le j_2-1$ and suppose that the game is played with the sink $L_j$. Since the maximum total degree (and thus, the maximum out-degree) of vertices in $L_{j-1}$ is at most $w_{\ell_{j-1}}$, the greens can easily win when the game is played with parameter $c=w_{\ell_{j-1}}$. They can play arbitrarily at the beginning of the game when the sludge is moving towards the sink. Once he reaches a vertex $u \in L_{j-1}$, the greens can block all out-neighbours and the game ends. We obtain that $g_j \le w_{\ell_{j-1}}$.

In order to derive a lower bound, we will need the following property that follows directly from the proof of Lemma~\ref{lem:max_deg} and holds a.a.s. Let us note that for $j_1 < j \le j_2$ we have that
\begin{eqnarray}\label{eq:cond_ell}
\ell_j &=& (1+o(1)) \ell_{j-1} d \left( \frac {n}{\ell_{j-1}} \right)^{\frac {1}{\beta-1}} \nonumber \\
&\ge& (1+o(1)) \ell_{j-1} d (\omega \log \log n)^{\frac {\max \{2, (\beta-1)^2 \}}{\beta-1}}.
\end{eqnarray}
For $j \le j_1$ we have  $\ell_j = \Omega(\ell_{j-1} n^{\alpha})$ so in fact~(\ref{eq:cond_ell}) holds for any $j \le j_2$.

Now let us play the game with parameter $c=w_{\ell_{j}} (1-\eps)$ for some $\eps >0$. We will show that a.a.s.\ the sludge can win the game, independently of the strategy of the greens. This will prove that $g_j \ge w_{\ell_{j}} (1-\eps)$ a.a.s.\ and the result will hold after taking $\eps \to 0$. If $j=1$, then for any $\alpha \in (0,1)$, a.a.s.  $|L_1|=M(1+o(1))=w_{\ell_1}(1+o(1))$, and thus, the greens clearly cannot win the game by protecting $w_{\ell_1}(1-\eps) = M(1-\eps)(1+o(1))$ vertices. Hence, we may assume that $j \geq 2$.

Suppose first that $\alpha \leq \frac{1}{\beta}$ and $2 \leq j \leq j_1-1$. Since in this case $j \leq \frac{1}{\alpha}-\beta+3$, by the formulas for $\ell_j$ given by Lemma~\ref{lem:max_deg}, for any $2 \leq j \leq j_1-1$, $\ell_j = \Theta(n^{\alpha(j-1)})$. Moreover, by part (v) of Lemma~\ref{lem:max_deg}, the number of edges between $L_{j-1}$ and $L_j$ is at most  $(1+o(1))\ell_{j+1}= O(n^{1-\alpha \beta+2\alpha}) = O(n^{1-\eps_0})$, with $\eps_0 = \alpha (\beta-2) > 0$. Hence, for any vertex $v \in L_{j}$,
$$
\Prob(\deg^-(v) \geq 2) \le (1+o(1)) {\ell_{j+1} \choose 2} \left( \frac 1n \right)^2 = O(n^{-2\eps_0}) \leq n^{-\eps_0} \leq n^{-\eps_1},
$$
where $\eps_1=\min \{\eps_0,\alpha/2\} > 0.$
Denoting by $B_v$ the number of out-neighbours of $v$ with in-degree $2$ or more, we have that $\E(B_v) \leq n^{\alpha-\eps_1}$. Since $O(n^{1-\eps_0})$ is a fixed upper bound on the number of edges between two consecutive layers, for any two vertices $v,v' \in L_{j}$, $\Prob(\deg^-(v) \geq 2 ~~|~~ \deg^-(v') \geq 2) \leq \Prob(\deg^-(v) \geq 2) \leq n^{-\eps_1}$.

Consider now another stochastic process in which each vertex $v \in L_{j}$, has exactly $n^{\alpha}$ out-neighbours, and for each out-neighbour $w$ of $v$, independently of all other vertices, $\Prob(\deg^-(w) \geq 2) = n^{-\eps_1}$. Denote by $B'_v$ the number of out-neighbours of $v$ with in-degree $2$ or more in this new stochastic process. Clearly, $\E(B'_v)=n^{\alpha-\eps_1}�\geq n^{\alpha/2}$ for any $v$. Furthermore, by the previous observation of negative correlation between vertices of in-degree $2$ or more (in the original process),  $\Prob(B_v \geq x) \leq \Prob(B'_v \geq x)$ for any $x \geq 0$. Set $\delta$ to be a sufficiently small constant. Then, by Lemma~\ref{lem:Chernoff}, for some $c>0$ we have that $$\Prob(B_v \geq (1+\delta)\E(B'_v)) \leq \Prob(B'_v \geq (1+\delta)\E(B'_v)) \leq  e^{-n^{c}}.$$  By taking a union bound over all $O(n)$ vertices of the first $j=O(1)$ layers, a.a.s.\ all vertices have at least a $1-o(1)$ fraction of out-neighbours with in-degree $1$.

Now, in order to show a lower bound on the green number, we can assume that all vertices with in-degree $2$ or more are already protected by the greens in the very beginning (they are cut away from top to bottom together with the subgraphs pending at them), and thus, the sludge is playing on the remaining graph that is a tree. We showed that a.a.s.\ the minimum degree in the remaining tree is at least $(1-o(1))w_{\ell_j}.$ Observe that in a tree the best strategy for the greens is always to protect neighbours of the vertex currently occupied by the sludge. Indeed, if they protect a vertex at distance $2$ or more from the vertex occupied by the sludge, they can consider the path between the vertex occupied by the sludge and the vertex originally protected, and instead protect the unique out-neighbour of the vertex occupied by the sludge. Clearly, this is at least as good move as the original one. Since the greens have only $w_{\ell_j}(1-\eps)$ at their disposal, in each round at least $\eps w_{\ell_j} - o(w_{\ell_j})$ neighbours remain unprotected, and the sludge can go to any of these, and finally reach the sink.

Suppose now that $\alpha > \frac{1}{\beta}$ or $j_1-1 < j \leq j_2-1$. For any $j_1 < j \leq j_2-1$ we have
\begin{eqnarray}\label{eq:cond_w}
w_{\ell_{j-1}} &=& (1+o(1)) c \ell_{j-1}^{-\frac {1}{\beta-1}} \ge (1+o(1)) c \ell_{j}^{-\frac {1}{\beta-1}} d^{\frac {1}{\beta-1}} (\omega \log \log n)^{\frac {\max \{2, (\beta-1)^2 \}}{(\beta-1)^2}} \nonumber \\
&\ge& (1+o(1)) w_{\ell_{j}} d^{\frac {1}{\beta-1}} (\omega \log \log n).
\end{eqnarray}
Moreover, note that the formula is true if $j> j_1-1$, but $j \leq j_1$, and also in the case $\alpha > \frac{1}{\beta}$ we have $j_1=2$, and $w_{\ell_1}=\Omega(w_{\ell_2}n^{\delta}) $ for some $\delta > 0$. Thus, combining these statements, for any $j_1-1 < j \le j_2-1$, or any $2 \leq j \leq j_2-1$ in the case $\alpha > \frac{1}{\beta}$, $w_{\ell_j} = o(w_{\ell_{j-1}} / \log \log n)$. Since $(w_i)_{i \geq 0}$ is a monotonically decreasing sequence, any vertex up to (and including) layer $L_{j-2}$ has weight at least $w_{\ell_{j-1}}=\omega(w_{\ell_j} \log \log n)=\omega(c \log \log n).$

We will provide a strategy for the sludge and show that it guarantees him win a.a.s., provided that the game is played with parameter $c = w_{\ell_{j}} (1-\eps)$, as before. The strategy is straightforward; in particular, he always goes to any non-protected vertex $u$ with the property that no out-neighbour of $u$ is protected. Note that the total number of vertices protected at the end of the game is $O(c j_2) = O(c \log \log n)$. Moreover, each protected vertex can only eliminate this vertex or its parents. The number of parents of a given vertex $u$, the in-degree of $u$, can be large but these parents are ``scattered'' across the whole layer, as we will show in the following claim.

\bigskip
\noindent \emph{Claim.} The following holds a.a.s. The number of paths from any vertex $v$ (a vertex possibly occupied by the sludge) and vertex $u$ two layers below (a vertex possibly protected by the greens) is bounded by some universal constant $K$.

\smallskip
\emph{Proof of the claim.} Fix $\eps > 0$ to be an arbitrarily small constant. Suppose first that $j$  is  such that $w_{\ell_{j-2}} < n^{\frac12 - \eps}$. Then, by monotonicity of $w_{\ell_j}$, the number of directed paths of length two starting at $v \in L_{j-2}$ is at most $(n^{\frac12-\eps})^2=n^{1-2\eps}$. Thus, the probability that for a given vertex $u \in L_j$ there are $K$ paths of length two starting from $v$, is at most $\binom{n^{1-2\eps}}{K}(\frac{1}{n})^{K}$. By taking a union bound over all $n$ vertices $u$ from $S \supseteq L_j$ and all starting vertices $v$ (there are at most $n$ of them), we see that $n^2 \binom{n^{1-2\eps}}{K}(\frac{1}{n})^{K}=o(1/\log \log n)$ for a sufficiently large constant $K$. The claim then holds for this range of values of $j$, by taking a union bound over $O(\log \log n)$ possible values of $j$.

On the other hand, if $j$ is such that $w_{\ell_{j-2}} \geq n^{\frac12 - \eps}$, then by the formulas for $w_{\ell_j}$ given in Lemma~\ref{lem:max_deg} we see that $j-2=O(1)$ and so $j+1=O(1)$ as well. (Indeed, the exponent of $n$ in the formula for $w_{\ell_{j-2}}$ in Lemma~\ref{lem:max_deg}(iv) can be made arbitrarily small by taking a sufficiently large constant $j-2$.) Since $j+1=O(1)$, it follows from Lemma~\ref{lem:max_deg} that $\ell_{j+1} \leq n^{1-\eps_0}$ for some $\eps_0 > 0$, and as shown in part (v) of Lemma~\ref{lem:max_deg}, the number of edges between $L_{j-1}$ and $L_j$  is at most $(1+o(1))\ell_{j+1}=O(n^{1-\eps_0}).$ Thus, the number of paths of length two starting at $v \in L_{j-2}$ is at most $O(n^{1-\eps_0})$. Hence, as before, the probability that for a given vertex $u \in L_j$ there are $K$ paths of length two starting from $v$, is at most $\binom{O(n^{1-\eps_0})}{K}(\frac{1}{n})^{K}$, and as before, by taking a union bound over all $n^2$ pairs of vertices $u,v$ and over all $j=O(1)$, for $K$ sufficiently large, $n^2 \binom{O(n^{1-\eps_0})}{K}(\frac{1}{n})^{K}=o(1)$, and the claim follows.

\bigskip

Hence, by the claim, we obtain that the number of eliminated vertices is still $O(c \log \log n)$. Finally, since the degree of each vertex in layers up to and including the layer $L_{j-2}$ is $\omega( c \log \log n)$, the sludge can easily reach the layer $L_{j-1}$. Since the minimum degree in this layer is $(1+o(1))w_{\ell_j}$ and the game is played with parameter $c=w_{\ell_{j}} (1-\eps)$', no matter what the greens do in this very last move, the sludge reaches the sink. Thus, $g_j \ge w_{\ell_{j}} (1-\eps)$ a.a.s. As we already mentioned, the result follows by taking $\eps$ tending to zero.
\end{proof}

Theorem~\ref{thm:main_power-law} follows immediately from Lemma~\ref{lem:max_deg} and Lemma~\ref{lem:green_power_law}.

\bigskip

We finish by remarking on how we can try to close the gap in the previous lemma.  It follows from Lemma~\ref{lem:green_power_law} that for $1 \leq  j \le j_2-1$ we have that a.a.s.\ $(1-o(1)) w_{\ell_{j}} \le g_j \le w_{\ell_{j-1}}$. Suppose then that the game is played with parameter $c$ such that $(1+\eps) w_{\ell_{j}} \le c \le (1-\eps) w_{\ell_{j-1}}$ for some $\eps > 0$. Clearly, the sludge tries to stay on vertices with as small label as possible (that is, the largest possible total degree). The greens aim for the opposite, they want the sludge to go to large labels (the smallest total degree). In the first round, the sludge is guaranteed to be able to go to a vertex with label at most $c+2$ and, in fact, the greens can force him to go to $v_{c+2}$. In the next round, it might be the case that there are some ``shortcuts'' to vertices of degree at least 2 with small labels but, since the number of such edges is very small, the greens can easily prevent the sludge from using these edges. After securing these edges, the greens should protect the remaining neighbours of $v_{c+2}$ with small labels. However, this time this does not help much. The sludge is forced to (but also is able to) go to a vertex  whose label is
$$
(1+o(1)) \sum_{i=1}^{c+2} (w_i-1) + O(c) = (1+o(1)) \sum_{i=1}^{c+2} w_i.
$$
Repeating this argument, and the calculations performed in the proof of Lemma~\ref{lem:max_deg}, we can compute the position of the sludge at time $j-1$ and based on that we can decide if he wins or looses this game. Optimizing this with respect to the parameter $c$ would yield the asymptotic value of $g_j$.

\section{Further directions}\label{directions}

We considered Seepage played on regular DAGs in Theorem~\ref{thm:main_d-reg}, and in power law DAGs in Theorem~\ref{thm:main_power-law}. It would be interesting to analyze the game on random DAGs with other degree sequences; for example, where the degree distribution remains the same at each level, or there are the same number of vertices at each level. While our emphasis was on asymptotic results for the green number in random DAGs, our results could be complemented by an analysis (via simulations) of the green number on small DAGs, say up to 100 vertices. We will consider such an approach in future work. Finally, hierarchical social networks are not usually strictly acyclic; for example, on Twitter, directed cycles of followers may occur. Seepage was defined in \cite{CFFMN} for DAGs, but it naturally extends to the setting with directed cycles (here, the directed graphs considered must have at least one source and a set of sinks; the game is then played analogously as before). A next step would be to extend our results, if possible, to a setting where such cycles occur, and analyze the green number on, say, their strongly connected components. One question is to determine if the green number change as a function of the number of backward edges.

\section{Acknowledgements} We would like to thank the anonymous referees for suggestions which improved the paper.


\begin{thebibliography}{99}

\bibitem{almen} J.A.\ Almendral, L.\ L\'{o}pez, and M.A.F.\ Sanju\'{a}n, Information flow in generalized hierarchical networks, \emph{Physica A} \textbf{324} (2003) 424--429.


\bibitem{bonato} A.\ Bonato, \emph{A Course on the Web Graph}, Graduate Studies in Mathematics Series, American Mathematical Society, Providence, Rhode Island, 2008.

\bibitem{AB1} A.~Bonato and R.J.\ Nowakowski, \emph{The Game of Cops and Robbers on Graphs}, American Mathematical Society, Providence, Rhode Island, 2011.

\bibitem{chayes} J.T.\ Chayes, B.\ Bollob\'{a}s, C.\ Borgs, and O.\ Riordan, Directed scale-free graphs, In: \emph{Proceedings of the 14th Annual ACM-SIAM Symposium on Discrete Algorithms}, 2003.

\bibitem{clbook} F.R.K.\ Chung and L.\ Lu, \emph{Complex graphs and networks}, American Mathematical Society, Providence RI, 2006.

\bibitem{CFFMN} N.E.\ Clarke, S.\ Finbow, S.L.\ Fitzpatrick, M.E.\  Messinger, and R.J.\ Nowakowski, Seepage in directed acyclic graphs, \emph{Australasian Journal of Combinatorics} \textbf{43} (2009)  91--102.

\bibitem{diestel} R.\ Diestel, \emph{Graph theory}, Springer-Verlag, New York, 2000.

\bibitem{farley2} J.D.\ Farley, Breaking Al Qaeda cells: a mathematical analysis of counterterrorism operations (A guide for risk assessment and decision making), \emph{Studies in Conflict \& Terrorism} \textbf{26} (2003) 399--411.

\bibitem{farley1} J.D.\ Farley, \emph{Toward a Mathematical Theory of Counterterrorism}, The Proteus Monograph Series Jonathan David Farley, Stanford University, 2007.

\bibitem{gupte} M.\ Gupte, S.\ Muthukrishnan, P.\ Shankar, L.\ Iftode, and J.\ Li, Finding hierarchy in directed online social networks, In: \emph{Proceedings of WWW'2011}.

\bibitem{interdiction} A.\ Gutfraind, A.\ Hagberg, and F.\ Pan, Optimal interdiction of unreactive Markovian evaders, In: \emph{Integration of AI and OR Techniques in Constraint Programming for Combinatorial Optimization Problems}, Hoeve, Willem-Jan van; Hooker, John N. (Eds), (Springer Berlin / Heidelberg) 2009.

\bibitem{ikeda} K.\ Ikeda and S.E.\ Richey, Japanese network capital: the impact of social networks on japanese political participation, \emph{Political behavior} \textbf{27} (2005) 239--260.

\bibitem{JLR} S. Janson, T. {\L}uczak, and A. Ruci\'nski, \emph{Random Graphs}, Wiley, New York, 2000.

\bibitem{lopez} L.\ L\'opez, J.F.F. Mendes, and M.A.F.\ Sanju\'an, Hierarchical social networks and information flow, \emph{Physica A: Statistical Mechanics and its Applications} \textbf{316} (2002) 695--708.

\bibitem{twitter} Twitaholic. Accessed January 10, 2012. \texttt{http://twitaholic.com/}.

\bibitem{west} D.B.\ West, \emph{Introduction to Graph Theory, 2nd edition}, Prentice Hall, 2001.

\bibitem{wormald} N.C.\ Wormald, Models of random regular graphs, \emph{Surveys in Combinatorics}, 1999, J.D.\ Lamb and D.A. Preece, eds.\ London Mathematical Society Lecture Note Series, vol 276, pp. 239--298. Cambridge University Press, Cambridge, 1999

\end{thebibliography}
\end{document}